\documentclass[11pt]{article}
\usepackage{booktabs}
\usepackage{caption}
\usepackage{mathrsfs}
\usepackage{amsmath}
\usepackage{amsfonts,amsthm,amssymb,mathrsfs,bbding}
\usepackage{float}
\usepackage{graphics,epstopdf}
\usepackage{graphics,multicol}
\usepackage{graphicx}
\usepackage{color}
\usepackage{enumerate}
\usepackage{tikz}
\usepackage{caption}
\usepackage{rotating}
\usepackage{lscape}
\usepackage{longtable}
\allowdisplaybreaks[4]
\usepackage{tabularx}
\usepackage[colorlinks=true,anchorcolor=blue,filecolor=blue,linkcolor=blue,urlcolor=blue,citecolor=blue]{hyperref}
\usepackage{extarrows}
\usepackage{cite}
\usepackage{latexsym,bm}
\usepackage{mathtools}

\evensidemargin=\oddsidemargin\topmargin=-1.5cm

\newtheorem{thm}{Theorem}[section]
\newtheorem{prob}[thm]{Problem}

\newtheorem{lem}{Lemma}[section]
\newtheorem{cor::3.1}{Corollary}[section]

\theoremstyle{definition}
\renewcommand\proofname{\bf Proof}

\addtocounter{section}{0}
\begin{document}
\title{\LARGE{\bf $l$‐connectivity, $l$‐edge‐connectivity and spectral radius of graphs}}
\author{Dandan Fan$^{a,b}$, Xiaofeng Gu$^c$, Huiqiu Lin$^{a,}$\thanks{Corresponding author.\\
{\it E-mail}: {\tt ddfan0526@163.com} (D. Fan), {\tt xgu@westga.edu} (X. Gu), {\tt huiqiulin@126.com} (H. Lin)}\\[2mm]
\small\it $^a$ School of Mathematics, East China University of Science and Technology, \\
\small\it   Shanghai 200237, China\\[1mm]
\small\it $^b$ College of Mathematics and Physics, Xinjiang Agricultural University,\\
\small\it Urumqi, Xinjiang 830052, China\\[1mm]
\small\it $^c$ Department of Computing and Mathematics, University of West Georgia, \\
\small\it Carrollton, GA 30118, USA
}
\date{}
\maketitle

\begin{abstract}
Let $G$ be a connected graph. The toughness of $G$ is defined as $t(G)=\min\left\{\frac{|S|}{c(G-S)}\right\}$, in which the minimum is taken over all proper subsets $S\subset V(G)$ such that $c(G-S)\geq 2$ where $c(G-S)$ denotes the number of components of $G-S$. Confirming a conjecture of Brouwer, Gu [\textit{SIAM J. Discrete Math.} 35 (2021) 948--952] proved a tight lower bound on toughness of regular graphs in terms of the second largest absolute eigenvalue. Fan, Lin and Lu [\textit{European J. Combin.} 110 (2023) 103701] then studied the toughness of simple graphs from the spectral radius perspective. While the toughness is an important concept in graph theory, it is also very interesting to study $|S|$ for which $c(G-S)\geq l$ for a given integer $l\geq 2$. This leads to the concept of the $l$-connectivity, which is defined to be the minimum number of vertices of $G$ whose removal produces a disconnected graph with at least $l$ components or a graph with fewer than $l$ vertices. Gu [\textit{European J. Combin.} 92 (2021) 103255] discovered a lower bound on the $l$-connectivity of regular graphs via the second largest absolute eigenvalue. As a counterpart, we discover the connection between the $l$-connectivity of simple graphs and the spectral radius. We also study similar problems for digraphs and an edge version.
\end{abstract}

\begin{flushleft}
\textbf{Keywords:} $l$‐connectivity; $l$‐edge‐connectivity; spectral radius; eigenvalue
\end{flushleft}
\textbf{MSC 2020:} 05C50, 05C40

\section{Introduction}
We consider simple graphs. Connectivity is a fundamental concept in graph theory. The \textit{connectivity} $\kappa(G)$ of a graph $G$ is the minimum number of vertices whose removal results in a disconnected graph or in the trivial graph. One related concept is the toughness, introduced by Chv\'atal \cite{Chva73} in 1973 to study cycle structures. The \textit{toughness} $t(G)$ of a connected non-complete graph $G$ is defined as $t(G)=\min\left\{\frac{|S|}{c(G-S)}\right\}$, in which the minimum is taken over all proper subsets $S\subset V(G)$ such that $G-S$ is disconnected and $c(G-S)$ denotes the number of components of $G-S$. Graph toughness has applications on hamiltonicity, pancyclicity, factors, spanning trees, and many others; see \cite{BaBS06} for a survey on related results. 

On the other hand, Chartrand et al.~\cite{CKLL} introduced the concept of $l$-connectivity, as a generalization of the classic connectivity. For an integer $l\geq 2$, the \textit{$l$‐connectivity} $\kappa_{l}(G)$ of a graph $G$ is defined to be the minimum number of vertices of $G$ whose removal produces a disconnected graph with at least $l$ components or a graph with fewer than $l$ vertices. Note that $\kappa_{2}(G)$ is the classic connectivity $\kappa(G)$. The $l$-connectivity is closely related to the toughness $t(G)$ of a graph~\cite{DOS}. In fact $t(G) =\min_{2\leq l\leq\alpha}\kappa_{l}(G)/l$, where $\alpha$ is the independence number of $G$. 
The $l$‐connectivity has been well studied for various kinds of graphs, such as, random graphs~\cite{GGSH}, pseudo-random graphs~\cite{Gu21}, hypercubes~\cite{ZYZ}, generalized Petersen graphs~\cite{Feeeero-Hanusch}, among others.  As a significant tool, $l$‐connectivity is extensively employed in the robustness analysis of computer networks (See\cite{ZY,LMS,LLFJC}).

\bigskip
This paper focuses on the research from spectral perspectives. For a graph $G$ on $n$ vertices, let $A(G)$ denote  the adjacency matrix of $G$ and let $\lambda_i:=\lambda_i(G)$ denote the $i$-th largest eigenvalue of $A(G)$ for $i=1,2,\ldots, n$. Specifically, the largest eigenvalue of $A(G)$ is called the \textit{spectral radius} of $G$ and is denoted by $\rho(G)$. Let $\lambda = \max_{2\le i\le n} |\lambda_i| =\max\{|\lambda_2|, |\lambda_n|\}$, that is, $\lambda$ is the second largest absolute eigenvalue. Denote by $D(G)$ the diagonal matrix of vertex degrees of $G$. The \textit{Laplacian matrix} of $G$ is defined as $L(G)=D(G)-A(G)$. Let $\mu_i:=\mu_i(G)$ denote the $i$-th smallest eigenvalue of the Laplacian matrix of $G$. The Laplacian matrix is positive semidefinite and so $\mu_{n}(G)\geq\ldots\geq\mu_{2}(G)\geq\mu_{1}(G)=0$. For a $d$-regular graph on $n$ vertices, $\lambda_i=d-\mu_{i}$ for $1\leq i\leq n$.

Over the past several decades, the connections between the connectivity of graphs and eigenvalues of associated matrices have been well investigated. One of the well-known results of Fiedler\cite{Fiedler} implies that $\kappa(G)\geq \mu_{2}(G)$ for a non-complete graph $G$. This result was improved for regular graphs by Krivelevich and Sudakov~\cite{Krivelevich-Sudakov} who proved that for a $d$-regular graph $G$ on $n$ vertices with $d\le n/2$, $\kappa(G)\ge d - 36\lambda^2/d$. Other related results have been obtained by Cioab\u{a} and Gu \cite{Cioba-Gu}, Abiad et al.~\cite{ABMOZ18}, Hong, Xia and Lai~\cite{HXL19}, and O\cite{SO-1}, among others.


The relation between the toughness and eigenvalues was first studied by Alon~\cite{Alon95} who showed that for any connected $d$-regular graph $G$, $t(G)>\frac{1}{3}(\frac{d^2}{d\lambda+\lambda^2}-1)$. Brouwer~\cite{Brou95} independently proved that $t(G)>\frac{d}{\lambda}-2$ for any connected $d$-regular graph $G$, and he conjectured that $t(G) \ge\frac{d}{\lambda}-1$ in \cite{Brou95, Brou96}. Some partial results can be found in \cite{CiWo14, Cioba-Gu, Gu21}. This conjecture was confirmed in~\cite{Gu21b} and has been extended to irregular graphs by using Laplacian eigenvalues in~\cite{GH22}. More recently, a spectral radius condition for toughness was provided in \cite{FLL23}, wherein the authors established spectral conditions for a graph to be 1-tough with minimum degree $\delta$ and to be $t$-tough, respectively.

Since the $l$-connectivity is considered as both a generalization of the connectivity and an extension of the toughness, motivated by the above spectral results, we are interested in discovering a spectral condition for $l$-connectivity. The second author~\cite{Gu21} ever proved a lower bound on the $l$-connectivity of regular graphs in terms of the second largest absolute eigenvalue. Here we intend to study maximum spectral radius of graphs of order $n$ with a given minimum degree $\delta$ and $l$-connectivity $\kappa_l$. Denote by $'\vee'$ and $'\cup'$ the join and union products, respectively. Recall that the $l$‐connectivity $\kappa_{l}(G)$ is the minimum number of vertices of $G$ whose removal produces a disconnected graph with at least $l$ components or a graph with fewer than $l$ vertices. If the latter case, then it implies that $n=\kappa_l +l-1$ and the graph with maximum spectral radius is obtained from $K_1\cup K_{\kappa_l +l-2}$ by adding $\delta$ edges in between. Therefore we may only consider the former case and thus $n\geq \kappa_l +l$.

It is known that $\kappa(G)\leq\delta(G)$. Nevertheless, when $l\geq 3$, $\kappa_{l}(G)$ and $\delta(G)$ are not bound by such a restriction. For $\kappa_l>\delta$, let $G^{\kappa_l,\delta}_{n}$ be the graph obtained from $K_{1}\cup (K_{\kappa_l}\vee (K_{n-\kappa_l-l+1}\cup (l-2)K_1))$ by adding $\delta$ edges between the isolated vertex $K_1$ and $K_{\kappa_l}$. If $\kappa_l\leq \delta$, let $G^{\kappa_l,\delta}_{n} = K_{\kappa_l}\vee(K_{n-\kappa_l-(l-1)(\delta-\kappa_l+1)}\cup(l-1)K_{\delta-\kappa_l+1})$.

We have the following theorem. In the rest of the paper, we always assume $l$ is an integer at least $2$, unless otherwise stated.

\begin{thm}\label{thm::1.1}
Let $G$ be a connected graph of order $n\geq \kappa_l +l$ with minimum degree $\delta$ and $l$-connectivity $\kappa_l$. Then $\rho(G)\leq \rho(G^{\kappa_l,\delta}_{n})$, with equality if and only if $G\cong G^{\kappa_l,\delta}_{n}$.
\end{thm}

It is worth mentioning that Theorem~\ref{thm::1.1} generalizes the result of Li et al.~\cite{LSCC} and independently Ye et al.~\cite{YFW10} who characterized the corresponding graph attaining the maximum spectral radius among all graphs with $\kappa(G)\leq k$ is $K_k\vee(K_1\cup K_{n-k-1})$, and the result of Lu and Lin~\cite{Lu-Lin} who proved that among all graphs with $\kappa(G)\leq k \leq \delta(G)$, the maximum spectral radius is obtained uniquely at $K_k\vee(K_{\delta-k+1}\cup K_{n-\delta-1})$.

\bigskip
We also consider an analogue for digraphs. The situation becomes more involved, since there are not many efficient tools for digraph spectra. For more about spectra of digraphs, we refer readers to the expository paper by Brualdi~\cite{Brualdi10}. 
Let $D=(V,E)$ be a digraph, where $V(D)$ and $E(D)$ are the vertex set and arc set of $D$, respectively. 
If $e=uv\in E(D)$, then $u$ is the initial vertex of $e$ and $v$ is the terminal vertex. We call $D$ \textit{strongly connected} if for every pair $x,y\in V(D)$, there exists a directed path from $x$ to $y$ and a directed path from $y$ to $x$. The \textit{complete digraph} $\overrightarrow{K}_{n}$ is a digraph of order $n$ in which every pair of distinct vertices $u,v$ is joined by an arc $uv$ and an arc $vu$. The \textit{$l$-connectivity} of a strongly connected digraph is the minimum number of vertices whose removal produces a digraph with at least $l$ strongly connected components or a digraph with fewer than $l$ vertices. For the later case, we can  deduce that $n=\kappa_l +l-1$ and the extremal graph attaining the maximum spectral radius is $\overrightarrow{K}_{\kappa_l +l-1}$. For the former case, let $\mathcal{D}_{n,\kappa_l}$ be the set of strongly connected digraphs of order $n\geq \kappa_l+l$ with $l$-connectivity $\kappa_l$. Lin et al.~\cite {LYZS} determined the maximum spectral radius among all digraphs in $\mathcal{D}_{n,\kappa_2}$ and characterized the corresponding spectral extremal graphs. We aim to solve the general problem in $\mathcal{D}_{n,\kappa_l}$ for an arbitrary integer $l\ge 2$. 

\begin{figure}[htb]
\centering
\begin{tikzpicture}[x=1.00mm, y=1.00mm, inner xsep=0pt, inner ysep=0pt, outer xsep=0pt, outer ysep=0pt, scale=0.6]
\path[line width=0mm] (39.94,19.30) rectangle +(120.88,84.72);
\definecolor{L}{rgb}{0,0,0}
\path[line width=0.30mm, draw=L] (98.00,89.58) ellipse (30.20mm and 12.44mm);
\definecolor{F}{rgb}{0,0,0}
\path[line width=0.30mm, draw=L, fill=F] (80.02,90.07) circle (1.00mm);
\path[line width=0.30mm, draw=L, fill=F] (89.06,90.09) circle (0.50mm);
\path[line width=0.30mm, draw=L, fill=F] (98.13,90.09) circle (0.50mm);
\path[line width=0.30mm, draw=L, fill=F] (107.46,90.09) circle (0.50mm);
\path[line width=0.30mm, draw=L, fill=F] (117.32,90.11) circle (1.00mm);
\path[line width=0.30mm, draw=L, fill=F] (59.92,60.09) circle (1.00mm);
\path[line width=0.30mm, draw=L, fill=F] (135.64,59.77) circle (1.00mm);
\path[line width=0.30mm, draw=L, fill=F] (90.65,32.63) circle (1.00mm);
\path[line width=0.30mm, draw=L, fill=F] (106.51,32.74) circle (1.00mm);
\path[line width=0.30mm, draw=L, fill=F] (94.95,32.68) circle (0.50mm);
\path[line width=0.30mm, draw=L, fill=F] (98.45,32.75) circle (0.50mm);
\path[line width=0.30mm, draw=L, fill=F] (102.28,32.77) circle (0.50mm);
\path[line width=0.30mm, draw=L] (79.91,89.97) -- (59.82,59.90);
\path[line width=0.30mm, draw=L] (79.91,90.16) -- (135.78,59.95);
\path[line width=0.30mm, draw=L] (117.38,90.16) -- (59.82,59.95);
\path[line width=0.30mm, draw=L] (117.25,90.16) -- (135.78,59.69);
\draw(128.82,89.92) node[anchor=base west]{\fontsize{9.23}{17.07}\selectfont $K_k$};
\draw(44.94,51.96) node[anchor=base west]{\fontsize{9.23}{17.07}\selectfont $K_{n_1}$};
\draw(93.01,22.38) node[anchor=base west]{\fontsize{9.23}{17.07}\selectfont $K_{n_2}$};
\draw(140.52,51.69) node[anchor=base west]{\fontsize{9.23}{17.07}\selectfont $K_{n_l}$};
\path[line width=0.30mm, draw=L] (59.94,53.05) ellipse (4.80mm and 11.28mm);
\path[line width=0.30mm, draw=L, fill=F] (60.01,56.49) circle (0.50mm);
\path[line width=0.30mm, draw=L, fill=F] (60.01,52.86) circle (0.50mm);
\path[line width=0.30mm, draw=L, fill=F] (60.01,49.23) circle (0.50mm);
\path[line width=0.30mm, draw=L, fill=F] (59.86,45.32) circle (1.00mm);
\path[line width=0.30mm, draw=L] (135.52,52.54) ellipse (4.21mm and 10.69mm);
\path[line width=0.30mm, draw=L, fill=F] (135.71,56.23) circle (0.50mm);
\path[line width=0.30mm, draw=L, fill=F] (135.84,52.73) circle (0.50mm);
\path[line width=0.30mm, draw=L, fill=F] (135.71,48.97) circle (0.50mm);
\path[line width=0.30mm, draw=L, fill=F] (135.84,45.21) circle (1.00mm);
\path[line width=0.30mm, draw=L] (80.10,89.89) -- (59.88,45.56);
\path[line width=0.30mm, draw=L] (117.30,90.12) -- (60.14,45.26);
\path[line width=0.30mm, draw=L] (80.10,90.12) -- (135.97,45.26);
\path[line width=0.30mm, draw=L] (117.30,89.86) -- (135.71,45.14);
\path[line width=0.30mm, draw=L] (109.53,35.73);
\path[line width=0.30mm, draw=L] (98.71,32.79) ellipse (12.10mm and 4.40mm);
\path[line width=0.30mm, draw=L] (80.19,89.99) -- (90.64,32.61);
\path[line width=0.30mm, draw=L] (80.01,89.99) -- (106.96,32.87);
\path[line width=0.30mm, draw=L] (117.41,90.18) -- (90.83,32.80);
\path[line width=0.30mm, draw=L] (117.59,89.99) -- (106.96,32.61);
\definecolor{L}{rgb}{1,0,0}
\path[line width=0.30mm, draw=L] (61.04,45.67) -- (134.64,59.69);
\definecolor{F}{rgb}{1,0,0}
\path[line width=0.30mm, draw=L, fill=F] (134.64,59.69) -- (131.76,59.86) -- (134.64,59.69) -- (132.02,58.48) -- (134.64,59.69) -- cycle;
\path[line width=0.30mm, draw=L] (61.13,45.58) -- (134.73,45.12);
\path[line width=0.30mm, draw=L, fill=F] (134.73,45.12) -- (131.94,45.84) -- (134.73,45.12) -- (131.93,44.44) -- (134.73,45.12) -- cycle;
\path[line width=0.30mm, draw=L] (61.13,60.15) -- (134.46,60.06);
\path[line width=0.30mm, draw=L, fill=F] (134.46,60.06) -- (131.66,60.76) -- (134.46,60.06) -- (131.66,59.36) -- (134.46,60.06) -- cycle;
\path[line width=0.30mm, draw=L] (61.04,59.97) -- (134.73,45.94);
\path[line width=0.30mm, draw=L, fill=F] (134.73,45.94) -- (132.11,47.16) -- (134.73,45.94) -- (131.85,45.78) -- (134.73,45.94) -- cycle;
\path[line width=0.30mm, draw=L] (60.76,59.79) -- (89.54,33.84);
\path[line width=0.30mm, draw=L, fill=F] (89.54,33.84) -- (87.93,36.24) -- (89.54,33.84) -- (86.99,35.20) -- (89.54,33.84) -- cycle;
\path[line width=0.30mm, draw=L] (61.13,59.69) -- (105.58,33.94);
\path[line width=0.30mm, draw=L, fill=F] (105.58,33.94) -- (103.51,35.95) -- (105.58,33.94) -- (102.81,34.73) -- (105.58,33.94) -- cycle;
\path[line width=0.30mm, draw=L] (91.65,32.51) -- (134.64,44.61);
\path[line width=0.30mm, draw=L, fill=F] (134.64,44.61) -- (131.76,44.53) -- (134.64,44.61) -- (132.14,43.18) -- (134.64,44.61) -- cycle;
\path[line width=0.30mm, draw=L] (107.33,33.24) -- (135.01,58.64);
\path[line width=0.30mm, draw=L, fill=F] (135.01,58.64) -- (132.47,57.26) -- (135.01,58.64) -- (133.42,56.23) -- (135.01,58.64) -- cycle;
\path[line width=0.30mm, draw=L] (91.65,33.24) -- (134.55,59.09);
\path[line width=0.30mm, draw=L, fill=F] (134.55,59.09) -- (131.79,58.25) -- (134.55,59.09) -- (132.51,57.05) -- (134.55,59.09) -- cycle;
\definecolor{L}{rgb}{0,0,0}
\definecolor{F}{rgb}{0,0,0}
\path[line width=0.30mm, draw=L, fill=F] (121.99,44.35) circle (0.50mm);
\path[line width=0.30mm, draw=L, fill=F] (127.12,48.93) circle (0.50mm);
\path[line width=0.30mm, draw=L, fill=F] (116.86,40.50) circle (0.50mm);
\definecolor{L}{rgb}{1,0,0}
\path[line width=0.30mm, draw=L] (60.94,44.97) -- (89.54,32.32);
\definecolor{F}{rgb}{1,0,0}
\path[line width=0.30mm, draw=L, fill=F] (89.54,32.32) -- (87.26,34.10) -- (89.54,32.32) -- (86.70,32.82) -- (89.54,32.32) -- cycle;
\path[line width=0.30mm, draw=L] (60.94,45.16) -- (104.94,33.06);
\path[line width=0.30mm, draw=L, fill=F] (104.94,33.06) -- (102.43,34.47) -- (104.94,33.06) -- (102.06,33.12) -- (104.94,33.06) -- cycle;
\path[line width=0.30mm, draw=L] (107.33,32.51) -- (136.11,44.24);
\path[line width=0.30mm, draw=L, fill=F] (136.11,44.24) -- (133.25,43.83) -- (136.11,44.24) -- (133.78,42.53) -- (136.11,44.24) -- cycle;
\end{tikzpicture}%
\caption{$\overrightarrow{G}^{k}_{n_1,n_2,\cdots,n_{l}}$.}\label{fig-1}
\end{figure}

Denote by $D_1\nabla D_2$ the digraph obtained from two disjoint digraphs $D_1, D_2$ with vertex set $V(D_1)\cup V(D_2)$ and arc set $E(D_1)\cup E(D_2)\cup \{uv, vu|~u\in V(D_1), v\in V(D_2)\}$. Let $\overrightarrow{G}^{k}_{n_1,n_2,\cdots,n_{l}}=\overrightarrow{K}_{k}\nabla (\overrightarrow{K}_{n_1}\cup \overrightarrow{K}_{n_2}\cup \cdots \cup \overrightarrow{K}_{n_l})\cup F$, where $F=\{vu|~v\in V(\overrightarrow{K}_{n_i}), u\in V(\overrightarrow{K}_{n_j})\}$ for all $1\leq i<j\leq l$ and $l\geq 2$, as shown in Figure 1. Particularly, let $\overrightarrow{G}_n^{\kappa_l}=\overrightarrow{G}^{\kappa_l}_{n_1,n_2,\cdots,n_{l}}$ with exactly one of $n_1, n_2, \cdots, n_{l}$ equals $n-\kappa_l -l+1$ and all other $n_i$'s are equal to $1$.

\begin{thm}\label{thm::1.2}
The digraph $\overrightarrow{G}^{\kappa_l}_{n}$ maximizes the spectral radius among all digraphs in $\mathcal{D}_{n,\kappa_l}$.
\end{thm}

Analogous to the $l$‐connectivity, Boesch and Chen~\cite{Boesch-Chen} defined the \textit{$l$‐edge‐connectivity} $\kappa'_{l}(G)$ of a graph $G$ on at least $l$ vertices to be the minimum number of edges whose removal leaves a graph with at least $l$ components. 
Clearly,  $\kappa'_{2}(G)$ is exactly the classic edge connectivity $\kappa'(G)$. The $l$-edge-connectivity is also related to edge-disjoint spanning trees. Let $\tau(G)$ denote the maximum number of edge-disjoint spanning trees in $G$, which is also called the \textit{spanning tree packing number} of $G$. From the spanning tree packing theorem by Nash-Williams~\cite{Nash-Williams} and independently by Tutte~\cite{Tutte61}, we have $\kappa'_l(G)\ge \tau(G) (l-1)$ and $\tau(G)=\min_{2\le l\le n} \kappa'_l(G)/(l-1)$.

The edge connectivity and the spanning tree packing number have been well studied from spectral perspectives. In particular, Chandran~\cite{Chandran}, Krivelevich and Sudakov~\cite{Krivelevich-Sudakov}, and Cioab\u{a}~\cite{Cioaba-1} discovered the connections between $\lambda_2$ and $\kappa'$ for regular graphs. Gu et al.~\cite{Gu-Lai} generalized Cioab\u{a}'s result to graphs that are not necessarily regular (see also \cite{Gu-1}). Most recently, O et al.\cite{SPPZ} provided a sharp result on edge connectivity of regular graphs in terms of $\lambda_2$. Cioab\u{a} and Wong~\cite{CiWo12} initiated the research on spanning tree packing number of regular graphs via $\lambda_2$ and made a conjecture. Gu et al.~\cite{Gu-Lai} partially solved the conjecture and extended it to general graphs. The conjecture was finally settled by Liu et al.~\cite{Liu-Hong} and the result was improved by Liu, Lai and Tian~\cite{Liu-Lai} with a Moore function.


Furthermore, Ning, Lu and Wang \cite{Ning} studied the edge connectivity with the spectral radius and proposed a conjecture: $B_{n,\delta+1}^{\kappa'}$ is the graph with the maximum spectral radius among all graphs of order $n$ with minimum degree $\delta$ and edge-connectivity $\kappa'$, where $B_{n,\delta+1}^{\kappa'}$ is a graph obtained from $K_{\delta+1}\cup K_{n-\delta-1}$ by adding $\kappa'$ edges joining a vertex in $K_{\delta+1}$ and $\kappa'$ vertices in $K_{n-\delta-1}$, and they verified this conjecture for the cases $0\leq \kappa'\leq 3$. Very recently, we confirmed this conjecture for $n\geq 2\delta+4$ in \cite{Fan-Gu-Lin}. In the same paper, we also investigated the spanning tree packing number and proved that for a connected graph $G$ with minimum degree $\delta\geq 2k\geq 4$ and order $n\geq 2\delta+3$, if $\rho(G)\geq \rho(B_{n,\delta+1}^{k-1})$, then $\tau(G)\geq k$ unless $G\cong B_{n,\delta+1}^{k-1}$.

Naturally, we will determine the graph(s) with the maximal spectral radius among all graphs with the given $l$‐edge‐connectivity. Let $H^{a,b}_{n}$ be a graph obtained from $K_{a-b+2}\vee(K_1\cup K_{n-a-1})$ by attaching $b-2$ pendent edges to some vertex in $K_{a-b+2}$, where $a\geq b-1$. Observe that $l\leq \kappa'_l+1$ is a trivial necessary condition for a connected graph with $l$‐edge‐connectivity $\kappa'_l$. We obtain the following result that generalizes the work of Ye et al.~\cite{YFW10} who proved the case $l=2$.

\begin{thm}\label{thm::1.3}
Let $G$ be a connected graph of order $n\geq 2\kappa'_l+2$ with $l$‐edge‐connectivity $\kappa'_l\geq l-1$. Then $$\rho(G)\leq \rho(H^{\kappa'_l,l}_{n}),$$
with equality if and only if $G\cong H^{\kappa'_l,l}_{n}$.
\end{thm}

\section{Proof of Theorem \ref{thm::1.1}}

For any $v\in V(G)$, let $N_{G}(v)$ and $d_G(v)$ be the neighborhood and degree of $v$, respectively. Particularly, let $N_{G}[v]=N_{G}(v)\cup\{v\}$.
\begin{lem}[See \cite{H.L-1}]\label{lem::2.2}
Let $G$ be a connected graph, and let $u,v$ be two vertices of $G$. Suppose that $v_{1},v_{2},\ldots,v_{s}\in N_{G}(v)\backslash N_{G}(u)$ with $s\geq 1$, and $G^*$ is the graph obtained from $G$ by deleting the edges $vv_{i}$ and adding the edges $uv_{i}$ for $1\leq i\leq s$. Let $x$ be the Perron vector of $A(G)$. If $x(u)\geq x(v)$, then $\rho(G)<\rho(G^*)$.
\end{lem}

\begin{lem}[See \cite{D.F}]\label{lem::2.3}
Let $n=\sum_{i=1}^t n_i+s$. If $n_{1}\geq n_{2}\geq \cdots\geq n_{t}\geq p$ and $n_{1}<n-s-p(t-1)$, then
$$\rho(K_{s} \vee (K_{n_{1}}\cup K_{n_{2}}\cup \cdots \cup K_{n_{t}}))<\rho(K_{s} \vee (K_{n-s-p(t-1)}\cup (t-1)K_{p})).$$
\end{lem}

For any vertex $v\in V(G)$ and any subset $S\subseteq V(G)$, let $N_{S}(v)=N_{G}(v) \cap S$ and $d_{S}(v)=|N_{G}(v) \cap S|$. Now we shall give a proof of Theorem \ref{thm::1.1}.
\renewcommand\proofname{\bf Proof of Theorem \ref{thm::1.1}}
\begin{proof}
Suppose that $G$ is a connected graph attaining the maximum spectral radius among all connected graph with $l$-connectivity $\kappa_l$ and minimum degree $\delta$. Since $n\geq \kappa_l +l$, by the definition of $l$-connectivity, there exists some nonempty subset $S\subseteq V(G)$ with $|S|=\kappa_{l}$ such that $G-S$ contains at least $l$ components. Let $C_1, C_2, \ldots, C_q$ be the components of $G-S$ where $q\geq l$, and let $|V(C_i)|=n_i$ for $1\leq i\leq q$.  Now, we divide the proof into the following two cases.

{\flushleft\bf{Case 1.}} $\kappa_l>\delta$.

Choose a vertex $u\in V(G)$ such that $d_{G}(u)=\delta$. Suppose that $d_{S}(u)=t$. 

{\bf Claim 1.} $u\notin S.$

{\bf Proof of Claim 1.}
Otherwise, $u\in S$. If $t=\delta$, then $N_{G}(u)\subseteq S$. Let $S'=S-u$. Then $|S'|=\kappa_l-1$ and $G-S'$ contains $q+1$ components $u, C_1, C_2,\cdots, C_q$, contradicting the $l$-connectivity $\kappa_l$ of $G$. Thus we may assume $t\leq \delta-1$. By the maximality of $\rho(G)$, we can obtain that $G-u\cong K_{\kappa_l-1}\vee (K_{n_1}\cup K_{n_2}\cup\cdots\cup K_{n_q})$. Note that $d_{S}(u)=t$ and $|N_{G}(u)\cap (V(G)\backslash S)|\geq 1$. Let $N_{S}(u)=\{v_1,v_2,\cdots,v_{t}\}$ and $S\backslash N_{S}[u]=\{v_{t+1},v_{t+2},\cdots,v_{\kappa_l-1}\}$, and let $P_{i}=N_{G}(u)\cap V(K_{n_i})$ and $|P_i|=p_{i}$ where  $1\leq i\leq q$. Then $\sum_{i=1}^{q}p_i=\delta-t$. Assume that $x$ is the Perron vector of $A(G)$. By symmetry, we say $x(v)=x_i$ for $ v\in P_{i}$ and $x(v)=x'_{i}$ for $v\in V(K_{n_i})\backslash P_i$ where $1\leq i\leq q$. Without loss of generality, we may assume that $x_{1}=\max\{x_i|~1\leq i\leq q\}$. Note that $x(v_{1})=x(v_{i})$ for $2\leq i\leq t$ and $x(v_{t+1})=x(v_{i})$ for $t+2\leq i\leq \kappa_l-1$. Then, from $A(G)x=\rho(G)x$, we have
\begin{eqnarray}
\rho(G)x_{i}&\!\!=\!\!&tx(v_1)+(\kappa_l-t-1)x(v_{t+1})+(p_i-1)x_i+(n_i-p_i)x'_{i}+x(u),\label{equ::1}\\
\rho(G)x'_{i}&\!\!=\!\!&tx(v_1)+(\kappa_l-t-1)x(v_{t+1})+p_ix_i+(n_i-p_i-1)x'_{i}, \label{equ::2}\\
\rho(G)x(u)&\!\!=\!\!&tx(v_1)+\sum_{i=1}^{q}p_ix_i,\label{equ::3}\\
\rho(G)x(v_{t+1})&\!\!=\!\!& tx(v_1)+(\kappa_l-t-2)x(v_{t+1})+\sum_{i=1}^{q}p_ix_i+\sum_{i=1}^{q}(n_i-p_i)x'_i,\label{equ::4}
\end{eqnarray}
where $1\leq i\leq q$. From (\ref{equ::1}), (\ref{equ::2}) and (\ref{equ::3}), we obtain
\begin{equation*}
\begin{aligned}
&~~~\rho(G)\Big(\sum_{i=2}^{q}p_ix_i+\sum_{i=2}^{q}(n_i-p_i)x'_i-x(u)\Big)\\
&=\Big(\sum_{i=2}^{q}n_i\!-\!1\Big)tx(v_1)\!+\!\sum_{i=2}^{q}
n_i(\kappa_l\!-\!t\!-\!1)x(v_{t+1})\!+\!\sum_{i=2}^{q}(n_i-1)p_ix_i+\sum_{i=2}^{q}(n_i-1)(n_i-p_i)x'_i+\\
&~~~\sum_{i=2}^{q}p_ix(u)-\sum_{i=1}^{q}p_ix_i\\
&\geq(q\!-\!2)tx(v_1)\!+\!(q\!-\!1)(\kappa_l\!-\!t\!-\!1)x(v_{t+1})\!-\!\sum_{i=1}^{q}p_ix_i~(\mbox{since $n_i\geq 1$ and $0\leq p_i\leq n_i$ for $2\leq i\leq q$})\\
&\geq(q-1)(\kappa_l-t-1)x(v_{t+1})-(\delta-t)x_1\\
&~~~(\mbox{since $q\geq l\geq 2$, $\sum_{j=1}^{q}p_j=\delta-t$ and $x_1\geq x_i$ for $2\leq i\leq q$})\\
&\geq(\delta-t)(x(v_{t+1})-x_1)~~(\mbox{since $q\geq l\geq 2$ and $\kappa_l\geq \delta+1$}).
\end{aligned}
\end{equation*}
Combining this with (\ref{equ::1}) and (\ref{equ::4}), we have
\begin{equation*}
\begin{aligned}
&(\rho(G)+1)(x(v_{t+1})-x_1)=\sum_{i=2}^{q}p_ix_i+\sum_{i=2}^{q}(n_i-p_i)x'_i-x(u)\geq\frac{(\delta\!-\!t)(x(v_{t+1})-x_1)}{\rho(G)},
\end{aligned}
\end{equation*}
from which we get
$$\frac{(\rho^2(G)+\rho(G)-\delta+t)(x(v_{t+1})-x_1)}{\rho(G)}\geq 0.$$
Notice that $G$ contains $K_{\kappa_l-1+n_i}$ as proper subgraphs, where $1\leq i\leq q$. Thus, $\rho(G)>\rho(K_{\kappa_l-1+n_i})=\kappa_l+n_i-2\geq \delta$ due to $n_i\geq 1$ for  $1\leq i\leq q$ and $\kappa_l\geq\delta+1$. Combining this with $t\geq 0$, we can deduce that $\rho^2(G)+\rho(G)-\delta+t>0$. This suggests that $x(v_{t+1})\geq x_1$.
Let $G_1=G-\{uv|~v\in P_i,~ 1\leq i\leq q\}+\{uv_j|~t+1\leq j\leq \delta\}$.
According to Lemma \ref{lem::2.2}, we have
\begin{equation}\label{equ::5}
\begin{aligned}
\rho(G_1)>\rho(G).
\end{aligned}
\end{equation}
Furthermore, we take a vertex $z\in V(K_{n_2})$. Let $G_2=G_1+\{wv|~w\in V(K_{n_1}), v\in V(K_{n_2})\backslash\{z\}\}+\{zv|~v\in V(G)\backslash(S\cup V(K_{n_2}))\}$, and let $S^*=S-u+z$. Then $|S^*|=\kappa_l$ and $G_2-S^*$ contains $q$ components $u,K_{n_1+n_2-1},K_{n_3},\cdots, K_{n_{q}}$. One can verify that $G_2$ is a connected graph with minimum degree $\delta$ and $l$-connectivity $\kappa_l$. Combining this with (\ref{equ::5}), we get
$$\rho(G_2)>\rho (G),$$
which contradicts the maximality of $\rho(G)$.
This implies that $u\notin S$, proving Claim 1.

\medskip
By Claim 1, $u\notin S$, and so $u\in V(C_i)$ for some $1\leq i\leq q$. Without loss of generality, we may assume that $u\in V(C_1)$.

{\bf Claim 2.} $N_{G}(u)\subseteq S$.

{\bf Proof of Claim 2.}
Otherwise, $N_{C_1}(u)\neq \varnothing$. Recall that $d_{S}(u)=t$. Then $d_{C_1}(u)=\delta-t\geq 1$. Let $N_{S}(u)=\{v_1,v_2,\cdots,v_{t}\}$,  $S\backslash N_{S}(u)=\{v_{t+1},v_{t+2},\cdots,v_{\kappa_l}\}$ and $N_{C_1}(u)=\{w_1,w_2,\cdots,w_{\delta-t}\}$. Again by the maximality of $\rho(G)$, we can deduce that $G-u\cong K_{\kappa_l}\vee (K_{n_1-1}\cup K_{n_2}\cup\cdots\cup K_{n_q})$. Let $x$ be the Perron vector of $A(G)$. By symmetry, we say $x_{i}=x(v)$ for $v\in V(C_i)$ where $2\leq i\leq q$, $x_1=x(v)$ for $v\in N_{C_1}(u)$, and $x'_1=x(v)$ for $v\in V(C_1)\backslash N_{C_1}[u]$.
Note that $x(v_{1})=x(v_{i})$ for $2\leq i\leq t$ and $x(v_{t+1})=x(v_{i})$ for $t+2\leq i\leq \kappa_l$. Then, from $A(G)x=\rho(G)x$, we have
\begin{eqnarray*}
\rho(G)x(v_{t+1})&\!\!=\!\!&tx(v_{1})+(\kappa_l-t-1)x(v_{t+1})+(\delta-t)x_1+(n_1-1-(\delta-t))x'_1+\sum_{i=2}^{q}n_ix_i,\\
\rho(G)x_1&\!\!=\!\!&tx(v_1)+(\kappa_l-t)x(v_{t+1})+(\delta-t-1)x_1+(n_1-1-(\delta-t))x'_1+x(u).
\end{eqnarray*}
Thus,
\begin{equation*}
\begin{aligned}
&~~~(\rho(G)+1)(x(v_{t+1})-x_1)\\
&=\sum_{i=2}^{q}n_ix_i-x(u)\\
 &=\frac{1}{\rho(G)}\Big(\sum_{i=2}^{q}n_i(\rho(G)x_i)-\rho(G)x(u)\Big)\\
 &=\frac{\sum_{i=2}^{q}n_{i}((n_i\!-\!1)x_i\!+\!tx(v_1)\!+\!(\kappa_l\!-\!t)x(v_{t+1}))\!-\!(tx(v_1)\!+\!(\delta\!-\!t)x_1)}{\rho(G)}\\
&\geq\frac{(\sum_{i=2}^{q}n_i\!-\!1)tx(v_1)\!+\!\sum_{i=2}^{q}n_i(\kappa_l\!-\!t)x(v_{t+1})\!-\!(\delta\!-\!t)x_1}{\rho(G)}~(\mbox{since $n_i\geq 1$ for $2\leq i\leq q$})\\
&\geq\frac{(q\!-\!2)tx(v_1)\!+\!(q\!-\!1)(\kappa_l\!-\!t)x(v_{t+1})\!-\!(\delta\!-\!t)x_1}{\rho(G)}~(\mbox{since $n_i\geq 1$ for $2\leq i\leq q$})\\
&>\frac{(\delta-t)(x(v_{t+1})-x_1)}{\rho(G)}~(\mbox{since $q\geq l\geq2$, $t\geq 0$, $\kappa_l>\delta$ and $\delta>t$}).\\
\end{aligned}
\end{equation*}
It follows that $\frac{(\rho^2(G)+\rho(G)-\delta+t)(x(v_{t+1})-x_1)}{\rho(G)}>0.$ Notice that $G$ contains $K_{\kappa_l}$ as a proper subgraph. Then $\rho(G)>\rho(K_{\kappa_l})=\kappa_l-1\geq\delta$, and hence $\rho^2(G)+\rho(G)-\delta+t>0$. This suggests that $x(v_{t+1})>x_1$. Take $G^*=G-\{uw_i|~ 1\leq i\leq \delta-t\}+\{uv_j|
~t+1\leq j\leq \delta\}$. It is easy to find that $G^*$ is also a connected graph with minimum degree $\delta$ and $l$-connectivity $\kappa_l$. By Lemma \ref{lem::2.2}, we have $\rho(G^*)>\rho(G)$, which contradicts the maximality of $\rho(G)$. Thus, we have $N_{G}(u)\subseteq S$, completing the proof of Claim 2. 

\medskip
In what follows, we shall prove that $G\cong G^{\kappa_l,\delta}_{n}$ for $\kappa_l>\delta$. In fact, by the maximality of $\rho(G)$, Claims 1 and 2, we can deduce that $G-u\cong K_{\kappa_l}\vee (K_{t_1}\cup K_{t_2}\cup\cdots\cup K_{t_{l-1}})$ for some integers $t_1\geq t_2\geq\cdots\geq t_{l-1}$, where $\sum_{i=1}^{l-1}t_i=n-\kappa_l-1$ and $N_{G}(u)\subseteq V(K_{\kappa_l})$. If $t_i=1$ for all $2\leq i\leq l-1$, then the result follows. If there exists  $t_j\geq 2$  for some $2\leq j\leq l-1$. Let $x(v)=x_1$ for $v\in V(K_{t_1})$ and $x(v)=x_2$ for $v\in V(K_{t_j})$. By $A(G)x=\rho(G)x$, we can deduce that
$$(\rho(G)-t_j+1)(x_1-x_2)=(t_1-t_j)x_1\geq 0$$
due to $t_1\geq t_j$. Since $G$ contains $K_{\kappa_l+t_j}$ as a proper subgraph and $\kappa_l\geq 1$, $\rho(G)>\rho(K_{\kappa_l+t_j})=\kappa_l+t_j-1\geq t_j$, and hence $x_1\geq x_2$. Take $w\in V(K_{t_j})$ and $z\in V(K_{t_1})$. Let $G'=G-\{wv|~v\in V(K_{t_j})\backslash \{w\}\}+\{zv|~v\in V(K_{t_j})\backslash \{w\}\}$. Then $G'$ is a connected graph with minimum degree $\delta$ and $l$-connectivity $\kappa_l$. According to Lemma \ref{lem::2.2}, we have $\rho(G')>\rho(G)$, which also leads to a contradiction. This implies that $G\cong G^{\kappa_l,\delta}_{n}$, as desired.

\medskip
{\flushleft\bf{Case 2.}} $\kappa_l\leq \delta$.

Recall that $|V(C_i)|=n_i$ for $1\leq i\leq q$. We assert that $q=l$. Otherwise, $q>l$. One can verify that $G$ is a proper spanning subgraph of $G''=K_{\kappa_l}\vee (K_{n_1}\cup \cdots \cup K_{n_{l-1}}\cup K_{n'_l})$ where $n'_l=\sum_{i=l}^{q}n_i$. It follows that $\rho(G)<\rho(G'')$, which contradicts the maximality of $\rho(G)$. This implies that $q=l$. Hence,
$$\rho(G)\leq \rho(K_{\kappa_l}\vee (K_{n_1}\cup K_{n_2}\cup \cdots \cup K_{n_l})),$$
with equality if and only if $G\cong K_{\kappa_l}\vee (K_{n_1}\cup K_{n_2}\cup \cdots \cup K_{n_l})$.
Without loss of generality, we may assume that $n_1\geq n_2\geq \cdots\geq n_{l}$. Clearly, $n_{l}\geq \delta+1-\kappa_l$ because the minimum degree of $G$ is $\delta$. Combining this with  Lemma \ref{lem::2.3}, we have
$$\rho(K_{\kappa_l}\vee (K_{n_1}\cup K_{n_2}\cup \cdots \cup K_{n_l}))\leq \rho(K_{\kappa_l} \vee (K_{n-\kappa_l-(\delta+1-\kappa_l)(l-1)}\cup (l-1)K_{\delta+1-\kappa_l})),$$
with equality if and only if $(n_1,n_2,\cdots,n_l)=(n-\kappa_l-(\delta+1-\kappa_l)(l-1),\delta+1-\kappa_l,\cdots,\delta+1-\kappa_l)$.
By the maximality of $\rho(G)$, we conclude that $G\cong K_{\kappa_l} \vee (K_{n-\kappa_l-(\delta+1-\kappa_l)(l-1)}\cup (l-1)K_{\delta+1-\kappa_l})$.

This completes the proof.
\end{proof}

\section{Proof of Theorem \ref{thm::1.2}}

Let $M$ be a real $n\times n$ matrix, and let $X=\{1,2,\ldots,n\}$. Given a partition $\Pi=\{X_1,X_2,\ldots,X_k\}$ with $X=X_{1}\cup X_{2}\cup \cdots \cup X_{k}$, the matrix $M$ can be partitioned as
$$
M=\left(\begin{array}{ccccccc}
M_{1,1}&M_{1,2}&\cdots &M_{1,k}\\
M_{2,1}&M_{2,2}&\cdots &M_{2,k}\\
\vdots& \vdots& \ddots& \vdots\\
M_{k,1}&M_{k,2}&\cdots &M_{k,k}\\
\end{array}\right).
$$
The \textit{quotient matrix} of $M$ with respect to $\Pi$ is defined as the $k\times k$ matrix $B_\Pi=(b_{i,j})_{i,j=1}^k$ where $b_{i,j}$ is the  average value of all row sums of $M_{i,j}$.
The partition $\Pi$ is called \textit{equitable} if each block $M_{i,j}$ of $M$ has constant row sum $b_{i,j}$.
Also, we say that the quotient matrix $B_\Pi$ is \textit{equitable} if $\Pi$ is an equitable partition of $M$.

\begin{lem}[Godsil~\cite{Godsil}, p.78 ]\label{lem::3.1}
Let $M$ be a real symmetric matrix, and let $\lambda_{1}(M)$ be the largest eigenvalue of $M$. If $B_\Pi$ is an equitable quotient matrix of $M$, then the eigenvalues of  $B_\Pi$ are also eigenvalues of $M$. Furthermore, if $M$ is nonnegative and irreducible, then $\lambda_{1}(M) = \lambda_{1}(B_\Pi).$
\end{lem}

\begin{lem}[Bondy and Murty~\cite{Bondy}, p.173]\label{lem::3.2}
Let $D$ be an arbitrary strongly connected digraph with vertex connectivity $k$. Suppose that $S$ is a $k$-vertex cut of $D$ and $D_1,\ldots, D_s$ are the strongly connected components of $D-S$. Then there exists an ordering of $D_1,\ldots, D_s$ such that, for $1\leq i\leq s$ and $v\in V(D_i)$, every tail of $v$ in $D_1$, $D_2$,\ldots, $D_{i-1}$.
\end{lem}

Also, by the well-known Perron–Frobenius theorem (cf.\cite[Section 8.8]{Godsil-Royle}), we can easily deduce the following result.
\begin{lem}\label{lem::3.3}
Let $D$ be a strongly connected digraph and $D'$ be a proper subgraph of $D$. Then $\rho(D')<\rho(D)$.
\end{lem}

Recall that $\overrightarrow{G}^{k}_{n_1,n_2,\cdots,n_{l}}=\overrightarrow{K}_{k}\nabla (\overrightarrow{K}_{n_1}\cup \overrightarrow{K}_{n_2}\cup \cdots \cup \overrightarrow{K}_{n_l})\cup F$, where $F=\{vu|~v\in V(\overrightarrow{K}_{n_i}), u\in V(\overrightarrow{K}_{n_j})\}$ for all $1\leq i<j\leq l$ and $l\geq 2$.
For $n_p,n_q\geq 2$,  $\overrightarrow{G}^{k}_{n_1,\cdots, n_{q}+1,\cdots, n_{p}-1,\cdots,n_{l}}$ is the graph obtained from $\overrightarrow{G}^{k}_{n_1,n_2,\cdots,n_{l}}$ by replacing $n_q$ and $n_p$ with $n_{q}+1$ and $n_{p}-1$, respectively.

\begin{lem}\label{lem::3.4}
Let $n_i$, $k$ and $l\geq 2$ be positive integers where $1\leq i\leq l$. If $n_q\geq n_p\geq 2$ for some $1\leq p,q\leq l$ and $p\neq q$, then  $\rho(\overrightarrow{G}^{k}_{n_1,n_2,\cdots,n_{l}})<\rho(\overrightarrow{G}^{k}_{n_1,\cdots, n_{q}+1,\cdots,n_{p}-1,\cdots,n_{l}})$.
\end{lem}
\renewcommand\proofname{\bf Proof}
\begin{proof}
Let $D=\overrightarrow{G}^{k}_{n_1,n_2,\cdots,n_{l}}$. Observe that $A(D)$ has the equitable quotient matrix
$$
A_{\Pi_1}=\begin{bmatrix}
k-1 &n_1  & n_2 &\cdots& n_l\\
k   &n_1-1&n_2& \cdots& n_l\\
k   &0    &n_2-1& \cdots& n_l\\
\vdots&\vdots   &\vdots& & \vdots\\
k   &0    &0    & \cdots& n_l-1\\
\end{bmatrix}.
$$
Then
\begin{equation*}
\varphi(A_{\Pi_1},\lambda)=\left|\begin{array}{cccccccc}
\lambda\!-\!(k\!-\!1)&-n_1&-n_2&\cdots&-n_{l-1}&-n_l\\
-k &\!\lambda\!-\!(n_1\!-\!1)\!&-n_2&\cdots&-n_{l-1}&-n_l\\
-k &0&\!\lambda\!-\!(n_2\!-\!1)\!&  \cdots& -n_{l-1}& -n_l\\
\vdots & \vdots& \vdots & &  \vdots& \vdots\\
-k & 0& 0& \cdots& \lambda\!-\!(n_{l-1}\!-\!1)& -n_l\\
-k & 0&  0&\cdots& 0& \lambda\!-\!(n_l\!-\!1)\\
\end{array}\right|
\end{equation*}
\begin{equation*}
~~~~~~~= \left|\begin{array}{cccccccc}
\lambda+1&-(\lambda+1)&0&\cdots&0&0\\
0 &\!\lambda\!-\!(n_1\!-\!1)\!&-(\lambda+1)&\cdots&0&0\\
0 &0&\!\lambda\!-\!(n_2\!-\!1)\!&  \cdots& 0& 0\\
\vdots & \vdots& \vdots & &  \vdots&\vdots\\
0 & 0& 0& \cdots& \lambda\!-\!(n_{l-1}\!-\!1)&-(\lambda+1)\\
-k & 0&  0&\cdots& 0& \lambda\!-\!(n_l\!-\!1)\\
\end{array}\right|
\end{equation*}
\begin{flalign*}
&\qquad\quad\quad\quad\:=(\lambda+1)\Big(\prod_{i=1}^{l}(\lambda-(n_i-1))-k(\lambda+1)^{l-1}\Big).&
\end{flalign*}
Assume that $D'=\overrightarrow{G}^{k}_{n_1,\cdots, n_{q}+1,\cdots, n_{p}-1,\cdots,n_{l}}$. Note that $A(D')$ has the equitable quotient matrix $B_{\Pi_2}$, which is obtained by replacing $n_q$ and $n_p$ with $n_{q}+1$ and $n_{p}-1$ in $A_{\Pi_1}$. Without loss of generality, we may assume that $n_t=\max\{n_1,\cdots,n_q+1,\cdots,n_p-1,\cdots,n_{l}\}$. Since $n_q\geq n_p$, it follows that $n_t\geq n_i$ where $1\leq i\leq l$. For
$\lambda\geq n_t$, we have
\begin{equation*}
\begin{aligned}
&~~\varphi(A_{\Pi_1},\lambda)-\varphi(B_{\Pi_2},\lambda)\\
&=\frac{(\lambda+1)\prod_{i=1}^{l}(\lambda-(n_i-1))}{(\lambda\!-\!(n_p\!-\!1))(\lambda\!-\!(n_q\!-\!1))}((\lambda\!-\!(n_p\!-\!1))(\lambda\!-\!(n_q\!-\!1))\!-\!
(\lambda\!-\!(n_p\!-\!2))(\lambda\!-\!n_q))\\
&=\frac{(\lambda+1)\prod_{i=1}^{l}(\lambda-(n_i-1))}{(\lambda\!-\!(n_p\!-\!1))(\lambda\!-\!(n_q\!-\!1))}(n_q-n_p+1)
\\
&>0 ~(\mbox{since $n_q\geq n_p\geq 2$}),
\end{aligned}
\end{equation*}
and hence $\varphi(A_{\Pi_1},\lambda)>\varphi(B_{\Pi_2},\lambda)$ for $\lambda\geq n_t$.  Notice that $D'$ is a strongly connected digraph and $\overrightarrow{K}_{k+n_t}$ is a proper subgraph of $D'$. Thus,
$\rho(D')>\rho(\overrightarrow{K}_{k+n_t})=n_t+k-1\geq n_t$ by Lemma \ref{lem::3.3} and $k\geq 1$. Combining this with Lemma \ref{lem::3.1}, we have
$$\rho(\overrightarrow{G}^{k}_{n_1,n_2,\cdots,n_{l}})<\rho(\overrightarrow{G}^{k}_{n_1,\cdots, n_{q}+1,\cdots, n_{p}-1,\cdots,n_{l}}),$$
as required. \end{proof}

Now, we will give a short proof of Theorem \ref{thm::1.2}.
\renewcommand\proofname{\bf Proof of Theorem \ref{thm::1.2}}
\begin{proof}
Suppose that $D$ is a digraph with maximum spectral radius in $\mathcal{D}_{n,\kappa_l}$. Thus, there exists some subset $S\subseteq V(D)$ with $|S|=\kappa_l$ such that $D-S$ contains at least $l$ strong connected components. By the maximality of $\rho(D)$, Lemmas \ref{lem::3.2} and \ref{lem::3.3}, we have $D\cong \overrightarrow{G}^{\kappa_l}_{n_1,n_2,\cdots,n_{l}}$ for some integers $n_i\geq 1$ ($1\leq i\leq l$).
Furthermore, according to Lemma \ref{lem::3.4}, we can deduce that $D\cong \overrightarrow{G}^{\kappa_l}_{n}$, as required.
\end{proof}

\section{Proof of Theorem \ref{thm::1.3}}

\begin{lem}[See \cite{Y.Hong}]\label{lem::4.1}
Let $G$ be a graph with $n$ vertices and $m$ edges. Then
                  $$\rho(G)\leq\sqrt{2m-n+1}.$$
\end{lem}

Let $\mathcal{K}_{n,n_1}^{k}$ be the set of connected graphs obtained from $K_{n-n_1}\cup n_1K_1$ by adding $k$ edges between $K_{n-n_1}$ and $n_1K_1$, where $k\geq n_1$. Let $e(G)$ denote the number of edges in $G$.
\begin{lem}\label{lem::4.2}
Let $G\in \mathcal{K}_{n,p-1}^{k}$ where $n\geq 2k+2$ and $2\leq p\leq k+1$. Then
$$n-p< \rho(G)< n-p+1.$$
\end{lem}

\renewcommand\proofname{\bf Proof}
\begin{proof}
Note that $G$ contains $K_{n-p+1}\cup (p-1)K_{1}$ as a proper spanning subgraph. Then $\rho(G)>\rho(K_{n-p+1}\cup (p-1)K_{1})= n-p$. On the other hand,
\begin{equation*}
\begin{aligned}
   e(G)={n-p+1\choose 2}+k.
\end{aligned}
\end{equation*}
Combining this with $n\geq 2k+2$, $2\leq p\leq k+1$ and Lemma \ref{lem::4.1}, we have
\begin{equation*}
\begin{aligned}
   \rho(G)\leq \sqrt{2e(G)-n+1}=\sqrt{(n-p+1)^2-(2n-p-2k)}<n-p+1.
\end{aligned}
\end{equation*}
This completes the proof.
\end{proof}

Recall that $H^{a,b}_{n}$ is the graph obtained from $K_{a-b+2}\vee(K_1\cup K_{n-a-1})$ by attaching $b-2$ pendent edges to some vertex in $K_{a-b+2}$, where $a\geq b-1$.
\begin{lem}\label{lem::4.3}
Let $G\in\mathcal{K}_{n,p-1}^{k}$ where $n\geq 2k+2$ and $k\geq p\geq 2$. Then
$$\rho(G)\leq \rho(H^{k,p}_{n}),$$
with equality if and only if $G\cong H^{k,p}_{n}$.
\end{lem}
\renewcommand\proofname{\bf Proof}
\begin{proof}
Suppose that $G'$ is a connected graph that attains the maximum spectral radius in $\mathcal{K}_{n,p-1}^{k}$ where $n\geq 2k+2$ and $k\geq p\geq 2$. For every $G\in \mathcal{K}_{n,p-1}^{k}$, we have
\begin{equation}\label{equ::8}
\begin{aligned}
\rho(G)\leq \rho(G').
\end{aligned}
\end{equation}
We partition $V(G')$ into $V_1\cup V_2$ with $V_1=V((p-1)K_{1})=\{v_1,v_2,\ldots,v_{p-1}\}$ and $V_2=V(K_{n-p+1})=\{u_{1},u_2,\ldots,u_{n-p+1}\}$. Let $x$ be the Perron vector of $A(G')$, and let $\rho=\rho(G')$. Without loss of generality, we assume that $x(v_{i+1})\leq x(v_{i})$ and $x(u_{j+1})\leq x(u_{j})$ for $1\leq i\leq p-2$ and $1\leq j\leq n-p$. Since $G'$ is a connected graph and $V_1$ is an independent set, it follows that $d_{G'}(v_i)=d_{V_2}(v_i)\geq 1$ for $1\leq i\leq p-1$. We first assert that $u_1\in N_{V_2}(v_i)$ for $1\leq i\leq p-1$. Otherwise, there exists some vertex, say $v_j$ ($1\leq j\leq p-1$), such that $u_1\notin N_{V_2}(v_j)$. Let $u_t\in N_{V_2}(v_j)$ where $2\leq t\leq n-p+1$, and let $G^*=G'-u_tv_j+u_1v_j$. Then $G^{*}\in\mathcal{K}_{n,p-1}^{k}$ and $\rho(G^{*})>\rho(G')$ by Lemma \ref{lem::2.2}, which contradicts the maximality of $\rho$. This implies that $u_1\in N_{V_2}(v_i)$ for all $1\leq i\leq p-1$. We next assert that $N_{G'}(v_{i+1})\subseteq N_{G'}(v_i)$ and $N_{G'}(u_{j+1})\subseteq N_{G'}[u_j]$ for $1\leq i\leq p-2$ and $1\leq j\leq n-p$. If not, suppose that there exist $i,j$ with $i<j$ such that $N_{G'}(v_j)\nsubseteq N_{G'}(v_i)$. Let
$w\in N_{G'}(v_j)\backslash N_{G'}(v_i)$ and $G^{**}=G'-wv_{j}+wv_{i}$. By considering that $u_1\in N_{G'}(v_i)\cap N_{G'}(v_j)$, we have $G^{**}\in\mathcal{K}_{n,p-1}^{k}$. Note that $x(v_{j})\leq x(v_{i})$. Then $\rho(G^{**})>\rho(G')$ by Lemma \ref{lem::2.2}, which contradicts the maximality of $\rho$. This implies that $N_{G'}(v_{i+1})\subseteq N_{G'}(v_i)$ for $1\leq i\leq p-2$. Similarly, we can deduce that $N_{G'}(u_{j+1})\subseteq N_{G'}[u_j]$ for $1\leq j\leq n-p$. Furthermore, we have
$N_{V_1}(u_{j+1})\subseteq N_{V_1}(u_j)$ for $1\leq j\leq n-p$.
Let $d_{V_2}(v_1)=t$. Again by the maximality of $\rho$ and Lemma \ref{lem::2.2}, we can obtain that $N_{V_2}(v_1)=\{u_1,u_2,\dots,u_t\}$. Recall that $d_{V_2}(v_i)\geq 1$ where $2\leq i\leq p-1$ and $\sum_{i=1}^{p-1}d_{V_2}(v_i)=k$. Then
$$t=d_{V_2}(v_1)=\sum_{i=1}^{p-1}d_{V_2}(v_i)-\sum_{i=2}^{p-1}d_{V_2}(v_i)\leq k-p+2.$$
If $t=k-p+2$, then $N_{V_2}(v_1)=\{u_1,u_2,\cdots, u_{k-p+2}\}$ and $N_{V_2}(v_i)=\{u_1\}$ for $2\leq i\leq p-1$, and hence $G'\cong H^{k,p}_{n}$. Combining this with (\ref{equ::8}), we can deduce that
$$\rho(G)\leq\rho(H^{k,p}_{n}),$$
with equality if and only if $G\cong H^{k,p}_{n}$, as required. Next we consider $t\leq k-p+1$ in the following. Note that $x(u_2)\geq x(u_i)$ for $3\leq i\leq t$, $x(u_{t+1})=x(u_{i})$ for $t+2\leq i\leq n-p+1$ and $x(v_2)\geq x(v_{i})$ for $3\leq i\leq p-1$. Then, by $A(G')x=\rho x$, we have
\begin{eqnarray*}
\rho x(u_{t+1}) =x(u_1)+\sum_{i=2}^{t}x(u_i)+(n-p-t)x(u_{t+1})\geq 2x(u_2)+(n-p-2)x(u_{t+1}).
\end{eqnarray*}
Note that $G'\in\mathcal{K}_{n,p-1}^{k}$. Then $\rho> n-p$ by Lemma \ref{lem::4.2}, and hence
\begin{eqnarray}
 x(u_{t+1})\geq\frac{2x(u_2)}{\rho-(n-p-2)}. \label{equ::9}
\end{eqnarray}
Assume that $E_1=\{v_1u_{i}|~t+1\leq i\leq k-p+2\}$ and $E_2=\{v_{i}u_j\in E(G') |~ 2\leq i\leq p-1, 2\leq j\leq t\}$. Let $G''=G'-E_2+E_1$. Clearly, $G''\cong H^{k,p}_{n}$. Let $y$ be the Perron vector of $A(G'')$, and let $\rho'=\rho(G'')$. By symmetry, we have $y(v_2)=y(v_i)$ for $3\leq i\leq p-1$, $y(u_2)=y(u_i)$ for $3\leq i\leq k-p+2$, and $y_1=y(u_i)$ for $k-p+3\leq i\leq n-p+1$. By $A(G'')y=\rho' y$, we have
\begin{eqnarray*}
\rho'y(v_2)&\!\!=\!\!&y(u_1),\\
\rho'y(v_1)&\!\!=\!\!&y(u_1)+(k-p+1)y(u_2),\\
\rho'y_1&\!\!=\!\!&y(u_1)+(k-p+1)y(u_2)+(n-k-2)y_1,
\end{eqnarray*}
from which we obtain that
\begin{equation}\label{equ::10}
\left\{
\begin{aligned}
  y(v_2)&=\frac{y(u_1)}{\rho'},\\
 y(u_2)&=\frac{\rho'y(v_1)-y(u_1)}{k-p+1},\\
 y_{1}&=\frac{\rho'y(v_1)}{\rho'-(n-k-2)}.
  \end{aligned}
\right.
\end{equation}
Notice that
\begin{eqnarray}
 \rho'y(u_1)&\!\!=\!\!&y(v_1)+(p-2)y(v_2)+(k-p+1)y(u_2)+(n-k-1)y_1.\label{equ::11}
\end{eqnarray}
Putting (\ref{equ::10}) into (\ref{equ::11}), we obtain that
\begin{eqnarray} \label{equ::12}
 y(v_1)=\frac{(\rho'^2+\rho'-(p-2))(\rho'-(n-k-2))}{\rho'(\rho'^2+2\rho'-(n-k-2))}y(u_1).
\end{eqnarray}
As $G',G''\in \mathcal{K}_{n,p-1}^{k}$, we get $\rho'\!-\!\rho>-1$ by Lemma \ref{lem::4.2}. Note that $x(v_2)\geq x(v_i)$ and $x(u_2)\geq x(u_j)$ for $3\leq i\leq p-1$ and $3\leq j\leq t$. Combining this with (\ref{equ::9}), (\ref{equ::10}), (\ref{equ::12}) and $\rho'\!-\!\rho>-1$, we have

\begin{equation*}
\begin{aligned}
&~~~y^{T}(\rho'-\rho)x \\
&=y^{T}(A(G'')-A(G'))x\\
   &=\sum_{v_1u_i\in E_1}(x(v_{1})y(u_{i})\!+\!x(u_{i})y(v_{1}))\!
-\sum_{v_iu_j\in E_2}\!(x(v_{i})y(u_{j})\!+\!x(u_{j})y(v_{i}))\\
   &\geq(k\!-\!p-t+2)(x(v_1)y(u_{2})+x(u_{t+1})y(v_1)-x(v_2)y(u_2)-x(u_2)y(v_2))\\
   &\geq(k\!-\!p-t+2)(x(u_{t+1})y(v_1)-x(u_2)y(v_2))~(\mbox{since $x(v_1)\geq x(v_2)$})\\
   &\geq(k\!-\!p-t+2)x(u_{2})y(u_1)\left(\frac{2(\rho'^{2}+\rho'\!-\!(p\!-\!2))(\rho'\!-\!(n\!-\!k\!-\!2))}
   {\rho'(\rho\!-\!(n\!-\!p\!-\!2))(\rho'^{2}\!+\!2\rho'\!-\!(n\!-\!k\!-\!2))} \!-\!\frac{1}{\rho'} \right)\\
   &~~(\mbox{by (\ref{equ::9}), (\ref{equ::10}) and (\ref{equ::12})})\\
   &=\frac{(k-p-t+2)x(u_{2})y(u_1)}{\rho'(\rho\!-\!(n\!-\!p\!-\!2))(\rho'^{2}\!+\!2\rho'\!-\!(n\!-\!k\!-\!2))}
   (2(\rho'^{2}+\rho'\!-\!(p\!-\!2))(\rho'\!-\!(n\!-\!k\!-\!2))\\
   &~~~-(\rho\!-\!(n\!-\!p\!-\!2))(\rho'^{2}\!+\!2\rho'\!-\!(n\!-\!k\!-\!2)))\\
        \end{aligned}
\end{equation*}
   \begin{equation*}
\begin{aligned}
    &=\frac{(k-p-t+2)x(u_{2})y(u_1)}{\rho'(\rho\!-\!(n\!-\!p\!-\!2))(\rho'^{2}\!+\!2\rho'\!-\!(n\!-\!k\!-\!2))}
   ((\rho'\!-\!(n\!-\!k\!-\!2))(\rho'^2\!+\!n\!-\!2p\!-\!k\!+\!2)\\
   &~~~\!+\!(\rho'\!-\!\rho\!+\!k\!-\!p\!+\!1)(\rho'^2\!+\!2\rho'\!-\!(n\!-\!k\!-\!2))\!-
   \!(\rho'^2\!+\!2\rho'\!-\!(n\!-\!k\!-\!2)))\\
   & >\frac{(k\!-\!p\!-\!t\!+\!2) [(\rho'\!-\!(n\!-\!k\!-\!2))(\rho'^2\!+\!n\!-\!2p\!-\!k\!+\!2)\!-\!(\rho'^2\!+\!2\rho'\!-\!(n\!-\!k\!-\!2))]x(u_{2})y(u_1)}{\rho'(\rho\!-\!(n\!-\!p\!-\!2))(\rho'^{2}\!+\!2\rho'\!-\!(n\!-\!k\!-\!2))}
  \\
   &~~~(\mbox{since $t\leq k-p+1$, $\rho'-\rho>-1$, $k\geq p$ and $\rho'>n-p\geq n-k$})\\
   &>\frac{(k\!-\!p\!-\!t\!+\!2)x(u_{2})y(u_1)}{\rho'(\rho\!-\!(n\!-\!p\!-\!2))(\rho'^{2}\!+\!2\rho'\!-\!(n\!-\!k\!-\!2))} [2(\rho'^2\!+\!n\!-\!2p\!-\!k\!+\!2)\!-\!(\rho'^2\!+\!2\rho'\!-\!(n\!-\!k\!-\!2))]
  \\
  &~~~(\mbox{since $\rho'>n-p\geq n-k$})\\
  &=\frac{(k\!-\!p\!-\!t\!+\!2)x(u_{2})y(u_1)}{\rho'(\rho\!-\!(n\!-\!p\!-\!2))(\rho'^{2}\!+\!2\rho'\!-\!(n\!-\!k\!-\!2))}
    (\rho'^2-2\rho'+3n-3k-4p+2)\\
   &\geq\frac{(k\!-\!p\!-\!t\!+\!2)x(u_{2})y(u_1)}{\rho'(\rho\!-\!(n\!-\!p\!-\!2))(\rho'^{2}\!+\!2\rho'\!-\!(n\!-\!k\!-\!2))}
    (\rho'^2\!-\!2\rho'\!-\!k\!+\!8)~~(\mbox{since $n\geq 2k+2$ and $k\geq p$})\\
    &> 0~~(\mbox{since $\rho>n-p$, $n\geq 2k+2$ and $\rho'> n-k\geq k+2$}),
\end{aligned}
\end{equation*}
and hence $\rho'>\rho$, which contradicts the maximality of $\rho$.
This completes the proof.
\end{proof}

\begin{lem}\label{lem::4.4}
Let $a$ and $b$ be two positive integers. If $a\geq b$, then
$${a\choose 2}+{b\choose 2}< {a+1\choose 2}+{b-1\choose 2}.$$
\end{lem}
\begin{proof}
Note that $a\geq b$. Then
\begin{equation*}
\begin{aligned}
&{a+1\choose 2}+{b-1\choose 2}\!-\!{a\choose 2}\!-\!{b\choose 2}=a-b+1>0.
\end{aligned}
\end{equation*}
Thus the result follows.
\end{proof}

For $X\subseteq V(G)$, let $G[X]$ be the subgraph of $G$ induced by $X$, and let $e(X)$ be the number of edges in $G[X]$.

\renewcommand\proofname{\bf Proof of Theorem \ref{thm::1.3}}
\begin{proof}
Suppose that $G$ is a connected graph of order $n\geq 2\kappa'_l+2$ attaining the maximum spectral radius among all
connected graphs with $l$‐edge‐connectivity $\kappa'_l$.
Thus, there exists some $E_1\subseteq E(G)$ with $|E_1|=\kappa'_{l}$ such that $G-E_1$ contains $p$ components $G_1, G_2,\cdots, G_p$ ($p\geq l$).
Observe that $H^{\kappa'_l,l}_{n}$ is a connected graph with $l$‐edge‐connectivity $\kappa'_l$ and $K_{n-l+1}$ is a proper subgraph of $H^{\kappa'_l,l}_{n}$. Thus,
\begin{eqnarray} \label{equ::13}
\rho(G)\geq \rho(H^{\kappa'_l,l}_{n})>\rho(K_{n-l+1})=n-l.
\end{eqnarray}
Let $|V(G_i)|=n_i$ for $1\leq i\leq p$.
Without loss of generality, we assume that $n_1\geq n_2\geq\cdots\geq n_{p}$. We first assert that $n_i=1$ for $2\leq i\leq p$. Otherwise, there exists some $j$ ($2\leq j\leq p$) such that $n_j\geq 2$. Hence, by Lemma \ref{lem::4.4}, we obtain that
$$e(G)=\sum_{i=1}^{p}e(G_i)+\kappa'_l\leq\sum_{i=1}^{p}{n_i\choose 2}+\kappa'_l\leq{n-p\choose 2}+{2\choose 2}+\kappa'_l={n-p\choose 2}+\kappa'_l+1.$$
Note that $G$ is connected. Then $p\leq \kappa'_l+1$. Combining this with Lemma \ref{lem::4.1}, $n\geq 2\kappa'_l+2$, $p\geq l$ and $\kappa'_l\geq 1$, we get
$$\rho(G)\leq \sqrt{2e(G)-n+1}=\sqrt{(n-p)^2-(2n-p-2\kappa'_l-3)}< n-p\leq n-l,$$
which contradicts (\ref{equ::13}). This implies that $n_1=n-p+1$ and $n_i=1$ for $2\leq i\leq p$. Suppose that $V(G_1)=\{u_1,u_2,\cdots,u_{n-p+1}\}$ and $V(G_i)=\{v_i\}$ where $2\leq i\leq p$.  By the maximality of $\rho(G)$, we can deduce that $G_1\cong K_{n-p+1}$. Let $x$ be the Perron vector of $A(G)$ and $V_2=\{v_2,v_3,\cdots,v_p\}$. Without loss of generality, we assume that $x(v_2)=\max\{x(v)|~v\in V_2\}$ and $x(u_1)\geq x(u_2)\geq\cdots\geq x(u_{n-p+1})$. Since $G$ is a connected graph, we have $e(V_2)+\sum_{i=3}^{p}d_{G_1}(v_i)\geq p-2$. Combining this with $e(V_2)+\sum_{i=2}^{p}d_{G_1}(v_i)=\kappa'_l$, we deduce that $d_{G_1}(v_2)\leq \kappa'_l-p+2$. By $A(G)x=\rho(G)x$, we have
\begin{eqnarray*}
\rho(G)x(u_{n-p+1})&\!\!\geq\!\!&\sum_{i=1}^{n-p}x(u_i),\\
\rho(G)x(v_2)&\!\! =\!\!& \sum_{v\in N_{V_2}(v_2)} x(v)+\sum_{v\in N_{G_1}(v_2)}x(v)\leq (p-2)x(v_2)+\sum_{i=1}^{\kappa'_l-p+2}x(u_i),
\end{eqnarray*}
from which we get
$$(\rho(G)\!-\!p\!+\!2)(x(u_{n-p+1})\!-\!x(v_2))\geq \sum_{i=1}^{n-2p+2}x(u_i)-\sum_{i=1}^{\kappa'_l-p+2}x(u_i)=\sum_{i=\kappa'_l-p+3}^{n-2p+2}x(u_i)
> 0 $$
due to $n\geq 2\kappa'_l+2$ and $p\leq \kappa'_l+1$. Since $G$ contains $K_{n-p+1}$ as a proper subgraph,  it follows that
$$\rho(G)>\rho(K_{n-p+1})=n-p\geq 2\kappa'_l+2-p\geq p,$$
where the inequalities follow from the fact that $n\geq 2\kappa'_l+2$ and $p\leq \kappa'_l+1$, and hence $x(u_{n-p+1})> x(v_2)$.
This suggests that $x(u_{i})> x(v_2)$ for $1\leq i\leq n-p+1$. We next assert that $u_1v_{j}\in E(G)$ for $2\leq j\leq p$. Otherwise,
let $U$ denote the set in $V_2$ consisting of all vertices that are not adjacent to $u_1$ and let $|U|=t$. Then $t\geq 1$.
As $G$ is connected, there exists $E_2\subseteq E(G)$ with $|E_2|=t$ formed by taking each vertex in $U$ as an endpoint.
Take $G^*=G-E_2+\{u_1v| v\in U\}$. It is easy to verify that $G^*$ is a connected graph with $l$-edge-connectivity $\kappa'_{l}$. Notice that $x(u_1)\geq x(v)$ for $v\in V(G)\backslash\{u_1\}$. Thus, $\rho(G^*)>\rho(G)$ by Lemma \ref{lem::2.2}, which also leads to a contradiction. This implies that $u_1\in N_{G_1}(v_j)$ for $2\leq j\leq p$. If $p=\kappa'_{l}+1$, then $G\cong H^{\kappa'_{l},p}_{n}$. Then we consider $p\leq \kappa'_{l}$ in the following. Since $x(u_i)> x(v_2)\geq x(v_j)$ for $1\leq i\leq n-p+1$ and $3\leq j\leq p$, again by the maximality of $\rho(G)$ and Lemma \ref{lem::2.2}, we can obtain that $e(V_2)=0$. It follows that $V_2$ is an independent set. Therefore, $G\in \mathcal{K}_{n,p-1}^{\kappa'_{l}}$. Combining this with Lemma \ref{lem::4.3}, we can deduce that $G\cong H^{\kappa'_{l},p}_{n}$. Furthermore, we claim that $p=l$. If not, $p\geq l+1$. Then $\rho(G)<n-p+1\leq n-l$ by Lemma \ref{lem::4.2}, which contradicts (\ref{equ::13}). This implies that $G\cong H^{\kappa'_{l},l}_{n}$, as required.

This completes the proof.\end{proof}

\section{Concluding remarks}

Theorem \ref{thm::1.3} indicates that $H^{\kappa'_l,l}_{n}$ is the unique graph that has the maximum spectral radius among all graphs of order $n\geq 2\kappa'_l+2$ with $l$‐edge‐connectivity $\kappa'_l$. Incorporating the minimum degree as a parameter naturally leads us to consider the following question.
\begin{prob}\label{prob::5.1}
Let $G$ be a connected graphs of order $n$ with $l$-edge-connectivity $\kappa'_l$ and minimum degree $\delta$. Determine the maximum spectral radius of $G$ and characterize the extremal graph(s).
\end{prob}
The results in \cite{Fan-Gu-Lin,Ning} provided the preliminary work of Problem \ref{prob::5.1} for $l=2$. However, the structure of connected graphs of order $n$ with $l$‐edge‐connectivity and minimum degree $\delta$ is more complicated for $l\geq 3$, and it seems difficult to determine the extremal graphs. This is left for possible further work.

The \textit{l-arc-connectivity} of a strongly connected digraph $D$ on at least $l$ vertices is defined as the minimum number of arcs are required to be deleted from $D$ to produce a digraph with at least $l$ strongly connected components. Lin and Drury \cite{Lin-Drury} characterized the extremal digraphs which attain the maximum spectral radius of digraphs with 2-arc-connectivity. Naturally, we consider the following question.

\begin{prob}\label{prob::5.2}
Let $l\geq 3$ and $D$ be a strongly connected digraph of order $n$ with a given $l$-arc-connectivity. Determine the maximum spectral radius of $D$ and characterize the extremal digraph(s).
\end{prob}

\section*{Acknowledgements}
Dandan Fan was supported by the National Natural Science Foundation of China (Grant No. 12301454) and sponsored by Natural Science Foundation of Xinjiang Uygur Autonomous Region(Grant Nos. 2022D01B103), Xiaofeng Gu was supported by a grant from the Simons Foundation (522728), and Huiqiu Lin was supported by the National Natural Science Foundation of China (Grant No. 12271162), Natural Science Foundation of Shanghai (No. 22ZR1416300) and The Program for Professor of Special Appointment (Eastern Scholar) at Shanghai Institutions of Higher Learning (No. TP2022031).


\begin{thebibliography}{99}
\setlength{\itemsep}{0pt}

\bibitem{ABMOZ18}
A. Abiad, B. Brimkov, X. Mart\' inez-Rivera, S. O, J. Zhang, Spectral bounds for the connectivity of regular graphs with given order, \textit{Electronic J. Linear Algebra} \textbf{34} (2018) 428--443.

\bibitem{Alon95}
N. Alon, Tough Ramsey graphs without short cycles, \textit{J. Algebraic Combin.} \textbf{4} (1995) 189--195.

\bibitem{BaBS06}
D. Bauer, H. J. Broersma, E. Schmeichel, Toughness of graphs--a survey, \textit{Graphs Combin.} \textbf{22} (2006) 1--35.

\bibitem{Boesch-Chen}
F.T. Boesch, S. Chen, A generalization of line connectivity and optimally invulnerable graphs, \textit{SIAM J. Appl. Math.} \textbf{34} (1978) 657--665.

\bibitem{Bondy}
J.A. Bondy, U.S.R. Murty, \textit{Graph Theory with Applications}, Macmillan, London, 1976.

\bibitem{Brou95} A. E. Brouwer, Toughness and spectrum of a graph, \textit{Linear Algebra Appl.} \textbf{226/228} (1995) 267--271.

\bibitem{Brou96} A. E. Brouwer, Spectrum and connectivity of graphs,\textit{ CWI Quarterly} \textbf{9} (1996) 37--40.

\bibitem{Brualdi10}
R.A. Brualdi, Spectra of digraphs, \textit{Linear Algebra Appl.} 432 (2010) 2181--2213.

\bibitem{CKLL}
G. Chartrand, S.F. Kapoor, L. Lesniak, D.R. Lick, Generalized connectivity in graphs, \textit{Bull. Bombay Math. Colloq.} \textbf{2} (1984) 1--6.

\bibitem{Chandran}
S.L. Chandran, Minimum cuts, girth and spectral threshold, \textit{Inform. Process. Lett.} \textbf{89} (2004) 105--110.

\bibitem{Chva73}
V. Chv\'atal, Tough graphs and hamiltonian circuits, \textit{Discrete Math.} \textbf{5} (1973) 215--228.

\bibitem{Cioaba-1}
S.M. Cioab\u{a}, Eigenvalues and edge-connectivity of regular graphs, \textit{Linear Algebra Appl.} \textbf{432}(1) (2010) 458--470.

\bibitem{Cioba-Gu}
S.M. Cioab\u{a}, X. Gu, Connectivity, toughness, spanning trees of bounded degree, and the spectrum of regular graphs, \textit{Czechoslovak Math. J.} \textbf{66}(3) (2016) 913--924.

\bibitem{CiWo12}
S.M. Cioab\u{a}, W. Wong, Edge-disjoint spanning trees and eigenvalues of regular graphs, \textit{Linear Algebra Appl.} \textbf{437} (2012) 630--647.

\bibitem{CiWo14} S.M. Cioab\u{a}, W. Wong, The spectrum and toughness of regular graphs, \textit{Discrete Appl. Math.} \textbf{176} (2014) 43--52.

\bibitem{DOS}
D.P. Day, O.R. Oellermann, H.C. Swart, Bounds on the size of graphs of given order and $l$-connectivity, \textit{Discrete Math.} \textbf{197/198} (1999) 217--223.

\bibitem{D.F}
D. Fan, S. Goryainov, X. Huang, H. Lin, The spanning $k$-trees, perfect matchings and spectral radius of graphs, \textit{Linear and Multilinear Algebra} \textbf{70} (2022) 7264--7275.

\bibitem{Fan-Gu-Lin}
D. Fan, X. Gu, H. Lin, Spectral radius and edge‐disjoint spanning trees, \textit{J. Graph Theory} (2023), \texttt{doi.org/10.1002/jgt.22996}.

\bibitem{FLL23}
D. Fan, H. Lin, H. Lu, Toughness, hamiltonicity and spectral radius in graphs,
\textit{European J. Combin.} \textbf{110} (2023) 103701.

\bibitem{Feeeero-Hanusch}
D. Ferrero, S. Hanusch, Component connectivity of generalized petersen graphs, \textit{Int. J. Comput. Math.} \textbf{91} (2014) 1940--1963.

\bibitem{Fiedler}
M. Fiedler, Algebraic connectivity of graphs, \textit{Czech. Math. J.} \textbf{23} (1973) 298--305.

\bibitem{Godsil}
D.C. Godsil, \textit{Algebraic Combinatorics}, Chapman \& Hall, New York, 1993.

\bibitem{Godsil-Royle}
D.C. Godsil, G. Royle, \textit{Algebraic Graph Theory}, New York, Springer-Verlag, 2001.

\bibitem{GGSH}
R. Gu, X. Gu, Y. Shi, H. Wang, $l$-connectivity and $l$-edge-connectivity of random graphs, \textit{J. Graph Theory} \textbf{101}(1) (2022) 5--28.

\bibitem{Gu-1}
X. Gu, Connectivity and spanning trees of graphs, PhD Dissertation, West Virginia University, 2013.


\bibitem{Gu21}
X. Gu, Toughness in pseudo-random graphs, \textit{European J. Combin.} \textbf{92} (2021) 103255.

\bibitem{Gu21b}
X. Gu, A proof of Brouwer's toughness conjecture, \textit{SIAM J. Discrete Math.} \textbf{35} (2021) 948--952.

\bibitem{GH22}
X. Gu, W.H. Haemers, Graph toughness from Laplacian eigenvalues, \textit{Algebraic Combinatorics} \textbf{5} (2022) 53--61.

\bibitem{Gu-Lai}
X. Gu, H.-J. Lai, P. Li, S. Yao, Edge-disjoint spanning trees, edge connectivity, and eigenvalues in graphs, \textit{J. Graph Theory} \textbf{81}(1) (2016) 16--29.


\bibitem{Y.Hong}
Y. Hong, A bound on the spectral radius of graphs, \textit{Linear Algebra Appl.} \textbf{108} (1988) 135--139.


\bibitem{HXL19}
Z.-M. Hong, Z.-J. Xia, H.-J. Lai, Vertex-connectivity and eigenvalues of graphs, \textit{Linear Algebra Appl.} \textbf{579} (2019) 72--88.

\bibitem{Krivelevich-Sudakov}
M. Krivelevich, B. Sudakov, Pseudo-random graphs, in: \textit{More sets, Graphs and Numbers}, Bolyai Soc. Math. Stud., vol 15, Springer, Berlin, 2006, 199--262.


\bibitem{LLFJC}
X. Li, C. Lin, J. Fan, X. Jia, B. Cheng, J. Zhou, Relationship between extra connectivity and component connectivity in networks, \textit{Comput. J.} \textbf{64}(1) (2021) 38--53.

\bibitem{LSCC}
J. Li, W.C. Shiu, W.H. Chan, A. Chang, On the spectral radius of graphs with connectivity at most $k$, \textit{J. Math. Chem.} \textbf{46} (2009) 340--346.

\bibitem{LYZS}
H. Lin, J. Shu, Y. Wu, G. Yu, Spectral radius of strongly connected digraphs, \textit{Discrete Math.} \textbf{312}(24) (2012) 3663--3669.

\bibitem{Lin-Drury}
H. Lin, S.W. Drury, The maximum Perron roots of digraphs with some given parameters, \textit{Discrete Math.} \textbf{313}(22) (2013) 2607--2613.

\bibitem{Liu-Hong}
Q. Liu, Y. Hong, X. Gu, H.-J. Lai, Note on edge-disjoint spanning trees and eigenvalues, \emph{Linear Algebra Appl.} \textbf{458} (2014) 128--133.

\bibitem{Liu-Lai}
R. Liu, H.-J. Lai, Y. Tian, Spanning trees packing number and eigenvalues of graphs with given girth, \textit{Linear Algebra Appl.} \textbf{578} (2019) 411--424.

\bibitem{H.L-1}
H. Liu, M. Lu, F. Tian, On the spectral radius of graphs with cut edges, \textit{Linear Algebra Appl.} \textbf{389} (2004) 139--145.

\bibitem{LMS}
X. Liu, J. Meng, E. Sabir, Component connectivity of the data center network DCell, \textit{Appl. Math. Comput.} \textbf{444} (2023) 127822.

\bibitem{Lu-Lin}
H. Lu, Y. Lin, Maximum spectral radius of graphs with given connectivity, minimum degree and independence number, \textit{J. Discrete Algorithms} \textbf{31} (2015) 113--119.

\bibitem{Nash-Williams}
C.St.J.A. Nash-Williams, Edge-disjoint spanning trees of finite graphs, \textit{J. Lond. Math. Soc.} \textbf{36} (1961) 445--450.

\bibitem{Ning}
W. Ning, M. Lu, K. Wang, Maximizing the spectral radius of graphs with fixed minimum degree and edge connectivity, \textit{Linear Algebra Appl.} \textbf{540} (2018) 138--148.

\bibitem{SO-1}
S. O, The second largest eigenvalue and vertex-connectivity of regular multigraphs, \textit{Discrete Appl. Math.} \textbf{279} (2020) 118--124.

\bibitem{SPPZ}
S. O, J.R. Park, J. Park, W. Zhang, Sharp spectral bounds for the edge-connectivity of regular graphs, \textit{European J. Combin.} \textbf{110} (2023) 103713.



\bibitem{Tutte61}
W.T. Tutte, On the problem of decomposing a graph into $n$ factors, \textit{J. Lond. Math. Soc.} \textbf{36} (1961) 221--230.

\bibitem{YFW10}
M. Ye, Y. Fan, H. Wang, Maximizing signless Laplacian or adjacency spectral radius of graphs subject to fixed connectivity, \textit{Linear Algebra Appl.} \textbf{433} (2010) 1180--1186.

\bibitem{ZY}
S. Zhao, W. Yang, Conditional connectivity of folded hypercubes, \textit{Discrete Appl. Math.} \textbf{257} (2019) 388--392.

\bibitem{ZYZ}
S. Zhao, W. Yang, S. Zhang, Component connectivity of hypercubes, \textit{Theor. Comput. Sci.} \textbf{640} (2016) 115--118.


\end{thebibliography}
\end{document}